\documentclass[11pt]{amsart}
\usepackage{amsmath,amssymb,amsfonts,amsthm,amscd,indentfirst}
\usepackage{amsmath,amssymb,amsfonts,amsthm,amscd}
\usepackage{indentfirst}
\usepackage{hyperref}
\usepackage{marginnote}
\usepackage{url}
\usepackage[toc,page]{appendix}
\numberwithin{equation}{section}

\newtheorem{prop}{Proposition}[section]
\newtheorem{theo}[prop]{Theorem}
\newtheorem{lemm}[prop]{Lemma}

\newtheorem{rema}[prop]{Remark}

\newtheorem{defi}[prop]{Definition}

\newtheorem{ques}[prop]{Question}

\def\and{\quad{\rm and}\quad}

\def\<{\langle}
\def\>{\rangle}

\usepackage{graphicx}

\title{A Liouville Theorem for M\"{o}bius Invariant Equations}

\begin{document}

\author[YanYan Li]{YanYan Li}
\address{Department of Mathematics, Rutgers University, Hill Center, Busch Campus, 110 Frelinghuysen Road, Piscataway, NJ 08854, USA.}
\email{yyli@math.rutgers.edu}

\author[Han Lu]{Han Lu}
\address{Department of Mathematics, Rutgers University, Hill Center, Busch Campus, 110 Frelinghuysen Road, Piscataway, NJ 08854, USA.}
\email{hl659@math.rutgers.edu}

\author[Siyuan Lu]{Siyuan Lu}
\address{Department of Mathematics and Statistics, McMaster University, 1280 Main Street West, Hamilton, ON, L8S 4K1, Canada.}
\email{siyuan.lu@mcmaster.ca}

\maketitle

\begin{abstract}
	In this paper we classify  M\"{o}bius invariant differential operators of second order in two dimensional Euclidean space, and establish a Liouville type theorem for general M\"{o}bius invariant elliptic equations.
\end{abstract}

\section{Introduction}

For $n\geq3$, consider the equation
\begin{align}\label{highdeqn1}
-\Delta u=n(n-2)u^{\frac{n+2}{n-2}}\text{ on }\mathbb{R}^n.
\end{align}
The Liouville type theorem of 
Caffarelli, Gidas and Spruck \cite{CGS} asserts that 
positive $C^2$ solutions of (\ref{highdeqn1}) are of the form
\begin{align*}
u(x)=(\frac{a}{1+a^2|x-\bar{x}|^2})^{\frac{n-2}{2}},
\end{align*}
where $a>0$ and $\bar{x}\in\mathbb{R}^n$.
Under an additional hypothesis
$u(x)=O(|x|^{2-n})$ 
for large $|x|$,  the result was established earlier 
by Obata \cite{Obata}  and Gidas, Ni and Nirenberg \cite{GNN}.

Geometrically, equation (\ref{highdeqn1}) means that the scalar curvature of the Riemannian metric $u^{ \frac {4}{n-2} } dx^2$ is equal to $4n(n-1)$.
An analogous equation in dimension two is
\begin{align}\label{2deqn1}
-\Delta u=e^u\text{ on }\mathbb{R}^2.
\end{align}
Geometrically, it means that the Gaussian curvature of $e^u dx^2$ is equal to $1/2$.
The above equation has plenty of solutions according to a
classical  theorem of Liouville \cite{Liou1}:
Let $\Omega$ be a simply connected domain in $\mathbb{R}^2$, then
all $C^2$ real solutions $u$ of  $-\Delta u=e^u$ in $\Omega$
are of the form
\begin{align} \label{rep}
u(x_1,x_2)=\ln \frac{8|f'(z)|^2}{(1+|f(z)|^2)^2},
\end{align}
where $z=x_1+\sqrt{-1}x_2$, and $f(z)$ is a locally univalent meromorphic function, i.e. a meromorphic function in $\Omega$ which has zeros or poles of order at most $1$. In particular, let $f(z)$ be any holomorphic function satisfying $f'(z)\ne 0$ in $\mathbb{C}$, the $u$ given by the  representation formula (\ref{rep}) is a solution of (\ref{2deqn1}). For instance, if we take $f(z)=e^z$, then we obtain a solution $u(x_1,x_2)=\log(8e^{2x_1}(1+e^{2x_1})^{-2})$.

On the other hand, Chen and Li proved in \cite{CL} that $C^2$ solutions of (\ref{2deqn1}) satisfying
\begin{align}\label{integrability}
\int_{\mathbb{R}^2} e^u\ dx<+\infty
\end{align}
are of the form
\begin{align*}
u(x)=2\ln \frac{8a}{8a^2+|x-x_0|^2}\ \ \ \mbox{in}\ 
 \mathbb{R}^2,
\end{align*}
where $a>0$ and $x_0\in \mathbb{R}^2$.

\medskip

Equation  (\ref{2deqn1}) is conformally invariant.
For  a $C^2$ function $u$, let
\begin{align}
u_\psi:=u\circ\psi+\ln |J_\psi|,
\end{align}
where $\psi(z)$  is a holomorphic function with nonzero Jacobian determinant $|J_\psi|$. 
Then we have
\begin{align*}
-e^{-u_\psi}\Delta u_\psi =(-e^{-u}\Delta u)\circ\psi\text{ on }\mathbb{R}^2.
\end{align*}
Here we consider the holomorphic function $\psi$ as a map from $\mathbb{R}^2$ to $\mathbb{R}^2$.

In particular, if $u$ is a solution of (\ref{2deqn1}), then $u_\psi$ is also a solution in the corresponding domain.
In fact, equation (\ref{2deqn1}) is in a sense the only conformally invariant equation as explained below.

Let 
$\mathcal{S}^{2\times 2}$ denote the set of $2\times 2$ real symmetric 
matrices.

\begin{defi}\label{conformalR2}
	Let $H$ be a function from $\mathbb{R}^2\times\mathbb{R}\times\mathbb{R}^2\times\mathcal{S}^{2\times 2}$ to $\mathbb{R}$, we say that a second order differential operator $H(\cdot, u, \nabla u, \nabla^2 u)$ is conformally invariant if for any meromorphic function $\psi$ on $\mathbb{C}$, and any function $u\in C^2(\mathbb{R}^2)$, it holds that:
   \begin{equation}\label{1}
   H(\cdot,u_{\psi},\nabla u_{\psi},\nabla^2 u_{\psi})\equiv H(\cdot,u,\nabla u,\nabla^2 u)\circ\psi.
   \end{equation}
\end{defi}

Note that (\ref{1}) is understood to hold at any point $z\in \mathbb{C}$ which is not a pole of $\psi$ or zero of $\psi'$.

The following proposition classifies all
conformally invariant  second order differentiable operators.

\begin{prop}\label{conformalinv}
$H(\cdot,u,\nabla u,\nabla^2 u)$ is conformally invariant in $\mathbb{R}^2$ if and only if it is of the form: 
\begin{align*}
H(\cdot,u,\nabla u,\nabla^2 u)=g(-e^{-u}\Delta u)
\end{align*}
where $g$ is a function from $\mathbb{R}$ to $\mathbb{R}$.
\end{prop}

\medskip

In this paper, we study a larger class of invariant operators,
namely those invariant under M\"{o}bius transformations.

 Recall that a M\"{o}bius transformation is a transformation generated by a finite composition of
\begin{align*}
\phi(x)=x+x_0,\ \phi(x)=\lambda x,\ \phi(x)=O x,\text{ and}\ \phi(x)=\frac{x}{|x|^2}\text{ in }\mathbb{R}^2
\end{align*}
where $\lambda$ is a nonzero constant, and $O$ is an orthogonal matrix. 
In complex variables $z=x_1+\sqrt{-1}x_2$, they are given by
$$
\phi(z)= \frac {az+b}{cz+d}\text{ or }\frac {a\bar z+b}{c\bar z+d},\ \ 
\ a,b,c,d \in \mathbb{C} \ \mbox{and}\ ad-bc\ne 0.
$$

\begin{defi}\label{MobiusR2}
        Let $H$ be a function from $\mathbb{R}^2\times\mathbb{R}\times\mathbb{R}^2\times\mathcal{S}^{2\times 2}$ to $\mathbb{R}$, we say $H(\cdot,u,\nabla u,\nabla^2 u)$ is M\"{o}bius invariant if for any M\"{o}bius transformation $\psi$ and any function $u\in C^2(\mathbb{R}^2)$, it holds that:
        \begin{align}\label{Mobiusfor}
        H(\cdot,u_{\psi},\nabla u_{\psi},\nabla^2 u_{\psi})
        =H(\cdot,u,\nabla u,\nabla^2 u)\circ\psi\qquad \text{on}\  \mathbb{R}^2.
        \end{align}
\end{defi}

We prove that all  M\"{o}bius invariant operators can be expressed as $F(A^u)$, 
where
\begin{equation}\label{Au}
A^u:=-e^{-u}\nabla^2u+\frac{1}{2}e^{-u}du\otimes du-\frac{1}{4}e^{-u}|\nabla u|^2I
\end{equation}
is a $2\times 2$ symmetric matrix operator of second order.

Notice that $tr(A^u)=-e^{-u}\Delta u$, conformally invariant operators are, as shown in Proposition \ref{conformalinv}, a function of the trace of $A^u$ in this setting.
 
The operator $A^u$ has the following invariance property:
For any M\"{o}bius transformation $\psi$ and any $x\in\mathbb{R}^2$, $A^{u_{\psi}}(x)=O^T(A^u\circ\psi)O$, where $O=|J_\psi|^{-1/2}J_\psi\in O(2)$, the set of all real orthogonal matrices.

%For example, if we take $\psi(x)=(x_1(x_1^2+x_2^2)^{-1},-x_2(x_1^2+x_2^2)^{-1})$, then
%\begin{align*}
%O=\frac{1}{x_1^2+x_2^2}\begin{pmatrix} x_2^2-x_1^2 & -2x_1x_2\\ 2x_1x_2   & x_2^2-x_1^2 \end{pmatrix}.
%\end{align*}
Moreover, for a M\"{o}bius transformation $\psi$, denote $y=\psi(x)$ as the coordinate change, and $v=u_\psi$, then
\begin{align*}
e^u(A^u_{kl}dy_k\otimes dy_l)=e^v(A^v_{ij}dx_i\otimes dx_j).
\end{align*}

Hence $F(A^{u_\psi})\equiv F(A^u\circ\psi)$ for any $F: \mathcal{S}^{2\times 2} \to \mathbb{R}$ which is invariant under orthogonal conjugation.  We say that $F$ is invariant under orthogonal conjugation if
\begin{equation}
F(O^{-1}MO)=F(M) \ \ \forall\ M\in \mathcal{S}^{2\times 2}, \
O\in O(2).
\label{I1}
\end{equation}

For $M\in  \mathcal{S}^{2\times 2}$, 
let $\lambda(M)=(\lambda_1(M), \lambda_2(M))$ with
 $\lambda_1(M)$ and $\lambda_2(M)$ being the eigenvalues of $M$. Then a function $F$ on $\mathcal{S}^{2\times 2}$ satisfying (\ref{I1}) corresponds to a symmetric function $f$ on $\mathbb{R}^2$ 
satisfying $F(M)=f(\lambda(M))$ for all $M\in\mathcal{S}^{2\times 2}$.

We classify all M\"{o}bius invariant operators in the following proposition.

\begin{prop}\label{Mobiusinv}
        Let $H(\cdot,u,\nabla u,\nabla^2 u)$ be M\"{o}bius invariant in $\mathbb{R}^2$, then $H$ is of the form:
        \begin{align*}
        H(\cdot,u,\nabla u,\nabla^2 u)=F(A^u)
        \end{align*}
        where $F:\mathcal{S}^{2\times 2}\rightarrow\mathbb{R}$ is
 invariant under orthogonal conjugation,  i.e. $F$ satisfies (\ref{I1}).
\end{prop}

In dimension $n\geq 3$, a classical theorem of Liouville states that any local conformal diffeomorphism in $\mathbb{R}^n$ is the restriction of a M\"{o}bius transformation. Therefore, unlike Definition \ref{conformalR2} and \ref{MobiusR2} for operators in $\mathbb{R}^2$, conformally invariant operators and M\"{o}bius invariant operators in $\mathbb{R}^n$ are the same for $n\geq 3$. The classification of conformally invariant  operators of second order was given by Li and Li in \cite{LL1}. Our proof of Proposition \ref{Mobiusinv} follows the arguments there.

\medskip

The main result in this paper is a Liouville type theorem for M\"{o}bius invariant elliptic equations $F(A^u)=1$ in $\mathbb{R}^2$.

From now on, let 
\begin{align}\label{gamma2}
\Gamma\text{ be an open convex symmetric}\text{ cone in }\mathbb{R}^2\text{ with vertex at the origin,}
\end{align}
and
\begin{align}\label{gamma}
\Gamma_2\subset \Gamma\subset\Gamma_1,
\end{align}
where $\Gamma_1:=\{(\lambda_1,\lambda_2):\lambda_1+\lambda_2>0\}$ 
and $\Gamma_2:=\{(\lambda_1,\lambda_2):\lambda_1>0, \lambda_2>0\}$. 
Here, $\Gamma$ being symmetric means that $(\lambda_1, \lambda_2)\in \Gamma$ 
implies $(\lambda_2, \lambda_1)\in \Gamma$.
Also, a function $f$ defined on $\Gamma$ is said to be symmetric if 
$f(\lambda_1, \lambda_2)\equiv f(\lambda_2, \lambda_1)$.

It is not difficult to see that $\Gamma$ satisfies  (\ref{gamma2}) and (\ref{gamma})  if 
and only if $\Gamma=\Gamma_p$ for some $1\le p\le 2$ where
\begin{equation*}
\Gamma_p= \{ \lambda=(\lambda_1, \lambda_2)\ :\
\lambda_2>(p-2)\lambda_1,\ 
\lambda_1>(p-2)\lambda_2\}.
\label{equiv}
\end{equation*}
Note that the above definition of $\Gamma_1$ and $\Gamma_2$ is consistent with 
earlier definitions.

\begin{theo}\label{nondege}
        Let $\Gamma=\Gamma_p$ for some $1<p\leq 2$,
and let $f\in C^1(\Gamma)$ be symmetric and satisfy $\partial_{\lambda_i}f> 0$ in $\Gamma$,  $i=1,2$. 
Assume that $u$ is a $C^2$ solution of
        \begin{align*}
        f\left(\lambda (A^u)\right)=1,\quad in \quad \mathbb{R}^2,
        \end{align*}
        where $\lambda(A^u)\in \Gamma$ are the eigenvalues of $A^u$.
        Then, for some $x_0\in \mathbb{R}^2$ and some constants $a,b>0$ satisfying $(2a^{-2}b,2a^{-2}b)\in \Gamma$ and $f(2a^{-2}b,2a^{-2}b)=1$,
        \begin{align}\label{solution}
        u(x)\equiv 2\ln\frac{8a}{8|x-x_0|^2+b}.
        \end{align}
\end{theo}
\begin{rema}
	Theorem \ref{nondege} still holds when replacing $f\in C^1(\Gamma)$ and $\partial_{\lambda_i}f> 0$ in $\Gamma$,  $i=1,2$ by: For any compact subset $K\subset\Gamma$, there exists constant $C(K)>1$ such that
	\begin{align*}
	\frac{1}{C(K)}\|\mu\|\leq f(\lambda+\mu)-f(\lambda)\leq C(K)\|\mu\|,\quad\forall\ \lambda,\lambda+\mu\in K,\mu_1,\mu_2>0,
	\end{align*}
\end{rema}
\begin{rema}
	For $u$ given in (\ref{solution}), $A^u=2a^{-2}bI$, where $I$ is the $2\times 2$ identity matrix.
\end{rema}
In the above theorem, no additional assumption on $u$ near infinity is made. In particular, the integrability condition (\ref{integrability}) is not assumed. The hypothesis $\partial_{\lambda_i}f>0$ means that the equation $f\left(\lambda (A^u)\right)=1$ is elliptic. For $\Gamma=\Gamma_1$ and $f(\lambda_1,\lambda_2)=\lambda_1+\lambda_2$, the equation is $-\Delta u=e^u$. As mentioned earlier, solutions were classified in \cite{CL} under the assumption (\ref{integrability}). In this case, the conclusion does not hold without the assumption. When $f(\lambda_1,\lambda_2)=\sigma_2(\lambda):=\lambda_1\lambda_2$, the equation becomes $\det(A^u)=1$.

A Liouville type theorem of the form $\sigma_2(\lambda(A^u))=1$ in $\mathbb{R}^4$ was proved by Chang, Gursky and Yang in \cite{CGY2}, where $\sigma_2$ denotes the second elementary symmetric function in $\mathbb{R}^4$. For $n\geq 3$, a Liouville type theorem for general elliptic equation $f(\lambda(A^u))=1$ in $\mathbb{R}^n$ was proved by Li and Li in \cite{LL1} and \cite{LL2}. The latter is an extension of the above mentioned Liouville type theorem for equation (\ref{highdeqn1}). The crucial point in the proof of Theorem \ref{nondege} is to handle the possible singularity of the solution at infinity. Our proof requires new ingredients, which enable us to establish appropriate asymptotic behavior of the solution near infinity.

\medskip

Chow and Wan  \cite[Corollary 2]{ChouW} gave a complex analysis proof of the above mentioned Liouville type theorem for equation (\ref{2deqn1}) by using Liouville's  representation formula (\ref{rep}).

\begin{ques}
	Is there a complex analysis proof of Theorem \ref{nondege}?
\end{ques}

In connection with the above question, we bring attention to a simple complex analysis proof of the following J\"orgen's theorem by Nitsche \cite{Nitsche}: Any smooth solution of $\det (D^2u)=1$ in $\mathbb{R}^2$ must be a quadratic polynomial.

We write $A^u_{ij}dx_i\otimes dx_j$ in complex variables. Denote $z=x_1+\sqrt{-1} x_2$ and 
$\bar z= x_1-\sqrt{-1} x_2$.
It 
is straightforward to check that
\begin{align*}
	A^u_{ij}dx_i\otimes dx_j=B^u_{zz}dz\otimes dz+B^u_{z\bar z}dz\otimes d\bar z+B^u_{\bar z z}d\bar z\otimes dz+B^u_{\bar z\bar z}d\bar z\otimes d\bar z,
\end{align*}

where
\begin{align*}
B^u=\begin{pmatrix} B^u_{z\bar z} & B^u_{zz} \\ B^u_{\bar z\bar z} & B^u_{\bar z z} \end{pmatrix}
=e^{-u}\begin{pmatrix} -u_{z\bar z} & -u_{zz}+\frac{1}{2}u_z^2 \\ -u_{\bar z\bar z}+\frac{1}{2}u_{\bar z}^2 & -u_{z\bar z} \end{pmatrix}.
\end{align*}

If $u$ is a real function, then $-u_{zz}+\frac{1}{2}u_z^2$ is the complex conjugate of $-u_{\bar z\bar z}+\frac{1}{2}u_{\bar z}^2$, and $B^u$ is a Hermitian matrix.

We have the following ways to describe the M\"{o}bius invariance in complex variables.

For any M\"{o}bius transformation $\psi$ and any $z\in\mathbb{C}$, $B^{u_{\psi}}(z)=U^*(B^u\circ\psi)U$, where
\begin{align*}
U=\frac{1}{\psi'(z)}\begin{pmatrix} \partial\psi/\partial z &\partial\psi/\partial\bar z \\  \partial\bar\psi/\partial z  & \partial\bar\psi/\partial\bar z \end{pmatrix}\in U(2),
\end{align*}
the set of all unitary matrices.

Moreover, for a M\"{o}bius transformation $\psi$, denote $z=\psi(w)$ as the coordinate change, and $v=u_\psi$. Then
\begin{align*}
&e^u(B^u_{zz}dz\otimes dz+B^u_{z\bar z}dz\otimes d\bar z+B^u_{\bar z z}d\bar z\otimes dz+B^u_{\bar z\bar z}d\bar z\otimes d\bar z)\\
=&e^v(B^v_{ww}dw\otimes dw+B^v_{w\bar w}dw\otimes d\bar w+B^v_{\bar w w}d\bar w\otimes dw+B^v_{\bar w\bar w}d\bar w\otimes d\bar w).
\end{align*}

\medskip

%For example, take $\psi(z)=\frac{1}{z}$, then
%\begin{align*}
%U=\begin{pmatrix} \frac{z^2}{\bar z^2} & 0\\ 0   & 1 \end{pmatrix}.
%\end{align*}

The second order differential matrix operator
$A^u$ in (\ref{Au}) corresponds to a  $(0,2)$-tensor on the standard Euclidean $2$-sphere $(\mathbb{S}^2, g_0)$ in $\mathbb{R}^3$. 
For a $C^2$ function $u$ on $\mathbb{S}^2$,
we associate with the conformal metric $g_u:=e^u g_0$ a $(0,2)$-tensor
\begin{equation*}
A_{g_u} := -\nabla_{g_0}^2u+\frac{1}{2}du\otimes du-
\frac{1}{4}|\nabla_{g_0} u|^2g_0+K_{g_0} g_0,
\label{A1} 
\end{equation*}
where $K_{g_0}\equiv 1 $ is the Gaussian curvature of the metric $g_0$.

Theorem \ref{nondege} is 
the starting point of our proof of the following results in a subsequent paper \cite{LLL} on the existence and compactness of solutions  to the  $\sigma_2$-Nirenberg problem.

For $K$ satisfying the nondegeneracy condition
\begin{align}\label{Knondege}
|\nabla K|_{g_0}+|\Delta K|_{g_0}>0 \text{ on }\mathbb{S}^2,
\end{align}
we define the sets
\begin{align*}
\text{Crit}_{+}(K)=\{x\in\mathbb{S}^2|\nabla_{g_0}K(x)=0,\Delta_{g_0}K(x)>0\},\\
\text{Crit}_{-}(K)=\{x\in\mathbb{S}^2|\nabla_{g_0}K(x)=0,\Delta_{g_0}K(x)<0\}.
\end{align*}
Set $\deg(\nabla K, \text{Crit}_{-}(K)):=\deg(\nabla K,O,0)$, where $O$ is any open subset of $\mathbb{S}^2$ containing $\text{Crit}_{-}(K)$ and disjoint from $\text{Crit}_{+}(K)$. By (\ref{Knondege}), this is well-defined.

\begin{theo}\label{Niren} (\cite{LLL})
        Assume that $K\in C^2(\mathbb{S}^2)$ is a positive function satisfying the nondegeneracy condition (\ref{Knondege}). Then there exists a positive constant $C$ depending only on $K$, such that
        \begin{align*}
        \|u\|_{C^2(\mathbb{S}^2)}\leq C,
        \end{align*}
        for all $C^2$ solutions $u$ of the equation
        \begin{align}\label{eqn}
        \sigma_2(\lambda(g_u^{-1}A_{g_u}))=K(x),\quad \lambda(A_{g_u})\in\Gamma_2\quad\text{on }\mathbb{S}^2.
        \end{align}
        Moreover, if $\deg(\nabla K, \text{Crit}_{-}(K))\neq 1$, then (\ref{eqn}) admits a solution.
\end{theo}
\begin{rema}
       If $K\in C^{2,\alpha}(\mathbb{S}^2)$, $0<\alpha<1$, and $\mathcal{O}$ is an open subset of $C^{4,\alpha}(\mathbb{S}^2)$ which contains all solutions of (\ref{eqn}), then
       \begin{align*}
       \deg(\sigma_2(\lambda(A^v))-K,\mathcal{O},0)=-1+\deg(\nabla K, \text{Crit}_{-}(K)).
       \end{align*}
       Here the degree on the left hand side is as defined in \cite{L89}.
\end{rema}
\begin{rema}
	For any $K$ satisfying (\ref{Knondege}) and having only isolated nondegenerate critical points, 
	\begin{align*}
	\deg(\nabla K, \text{Crit}_{-}(K))=\sum_{\bar{x}\in\mathbb{S}^2,\nabla K(\bar{x})=0,\Delta K(\bar{x})<0}(-1)^{i(\bar x)}
	\end{align*}
	where $i(\bar x)$ denotes the number of negative eigenvalues of $\nabla^2 K(\bar x)$.
\end{rema}
Theorem \ref{Niren} is related to the Nirenberg problem, which amounts to solving $\sigma_1(\lambda(A_{g_u}))=K$ on $\mathbb{S}^2$. There has been much work on the Nirenberg problem and related ones, see e.g. Chang and Liu \cite{CLiu}, Chang and Yang \cite{CY}, Han \cite{Han}, Jin, Li and Xiong \cite{JLX}, and the references therein. For $n\geq 3$ and $k\geq 2$, the $\sigma_k$-Nirenberg problem was studied by Chang, Han and Yang \cite{CHY} and Li, Nguyen and Wang\cite{LNW1}.

\medskip

The above mentioned Liouville type theorem for general conformally invariant equations $f(\lambda(A^u))=1$ in \cite{LL2} was stimulated by the study of fully nonlinear elliptic equations involving the Schouten tensor, and in particular by the study of the $\sigma_k$-Yamabe problem.

The existence of solutions of the $\sigma_k$-Yamabe problem has been proved for $k\geq n/2$, $k=2$ or when $(M,g)$ is locally conformally flat, and the compactness of the set of solutions has been proved for $k\geq n/2$ when the manifold is not conformally equivalent to the standard sphere $-$ they were established in \cite{CGY1,GeW,GW2,GV,LL1,LN2,STW}. For more recent works on $\sigma_k$-Yamabe type problems, see for example \cite{AbEs,BCE,BoSheng,Case,CaseW1,CaseW2,FangW,FangW2,GLN,GS,HLL,He,JS1,JS2,JS3,LN3,LN4,LW,Santos,Sui,T2} and references therein. On the other hand, the existence of solutions remains open when $2<k<n/2$, and the compactness of solutions remains open when $2\leq k<n/2$.

%In high dimensional cases, conformally invariant equations are helpful to understand the Yamabe problem and the Nirenberg problem in geometric analysis. 

One of our motivations in studying the M\"{o}bius invariant equations in dimension two is to gain insights and inspirations into solving the above mentioned open problems on the existence and compactness of the $\sigma_k$-Yamabe problem, for $2\leq k<n/2$.

\medskip

The rest of our paper is organized as follows. In Section \ref{viscosity}, we recall the definition of viscosity solutions for $\lambda(A^u)\in\partial\Gamma$ and a regularization lemma. We then give a proof of the crucial asymptotic behavior of viscosity supersolutions. Theorem \ref{nondege} is proved in Section \ref{proofLiou} using the method of moving spheres, a variant of the method of moving planes. Proposition \ref{conformalinv} and \ref{Mobiusinv} are proved in Section \ref{classi}. Two calculus lemmas are given in Appendix \ref{appendix} for the reader's convenience.

\section*{Acknowledgements}
The first named author's research was partially supported by NSF Grants DMS-1501004, DMS-2000261, and Simons Fellows Award 677077. The second named author's research was partially supported by NSF Grants DMS-1501004, DMS-2000261. The third named author's research was partially supported by NSERC Discovery Grant.

\section{Asymptotic behavior}\label{viscosity}

In this section, we establish an estimate on the asymptotic behavior for viscosity supersolution of $\lambda(A^u)\in\partial\Gamma$ on $\mathbb{R}^2\backslash\{\text{compact set}\}$. This estimate allows the method of moving spheres argument to get started in the proof of Theorem \ref{nondege} in section \ref{proofLiou}.

We start with the definition of viscosity solutions for $\lambda(A^u)\in\partial\Gamma$, see \cite{L09} and \cite{LNW} for details.

Let us first define the set of upper semi continuous and lower semi continuous functions.

For any set $S\subset \mathbb{R}^2$, we use $USC(S)$ to denote the set of functions $u:S\rightarrow \mathbb{R}\cup \{-\infty\}$, $u\neq -\infty$ in $S$, satisfying
\begin{align*}
\limsup_{x\rightarrow x_0}u(x)\leq u(x_0),\quad \forall x_0\in S.
\end{align*}

Similarly, we use $LSC(S)$ to denote the set of functions $u:S\rightarrow \mathbb{R}\cup \{+\infty\}$, $u\neq +\infty$ in $S$, satisfying
\begin{align*}
\liminf_{x\rightarrow x_0}u(x)\geq u(x_0),\quad \forall x_0\in S.
\end{align*}

\begin{defi}
Let $\Gamma$ satisfy (\ref{gamma2}) and (\ref{gamma}) and $\Omega$ be an open subset in $\mathbb{R}^2$, we say $u\in USC(\Omega)$ is a viscosity subsolution of 
\begin{align}\label{viscosity solution}
\lambda(A^u)\in \partial\Gamma,\quad in \quad \Omega
\end{align}
if for any point $x_0\in \Omega$, $\varphi\in C^2(\Omega)$, $(u-\varphi)(x_0)=0$ and $u-\varphi\leq 0$ near $x_0$, we have
\begin{align*}
\lambda(A^\varphi(x_0))\in \mathbb{R}^2\setminus \Gamma.
\end{align*}

Similarly, we say $u\in LSC(\Omega)$ is a viscosity supersolution of (\ref{viscosity solution}), if for any point $x_0\in \Omega$, $\varphi\in C^2(\Omega)$, $(u-\varphi)(x_0)=0$ and $u-\varphi\geq 0$ near $x_0$, we have
\begin{align*}
\lambda(A^\varphi(x_0))\in \bar{\Gamma}.
\end{align*}

We say $u$ is a viscosity solution of (\ref{viscosity solution}), if it is both a subsolution and a supersolution.
\end{defi}

\medskip

Let us recall the well-known regularization of semi-continuous functions which will be used in the paper, see \cite{CC} and \cite{LNW} for details.

\begin{lemm}\label{Jensen Approximation}
	Let $\Omega$ be an open set, $\Omega'\subset\subset\Omega$, and $\Omega''\subset\subset\Omega'$. Let $u\in C(\bar{\Omega})$ be a viscosity supersolution of (\ref{viscosity solution}), we define the $\epsilon$-lower envelope of $u$ by 
	\begin{align*}
	u_\epsilon(x):=\inf_{y\in \bar{\Omega}'}\{u(y)+\frac{1}{\epsilon}|y-x|^2\},\quad \forall x\in \bar{\Omega'}.
	\end{align*}
	
	Then there exists $\epsilon_0=\epsilon_0(\Omega',\Omega'')\ >0$, such that for $0<\epsilon<\epsilon_0$,
	
	(i) $u_\epsilon\rightarrow u$ in $C^0(\Omega'')$, as $ \epsilon\rightarrow 0^+$.
	
	(ii) $u_\epsilon(x)-\frac{1}{\epsilon}|x|^2$ is concave in $\Omega'$. Consequently, $u_\epsilon$ is second order differentiable almost everywhere in $\Omega''$ and $u_\epsilon$ is Lipschitz in $\Omega''$.

	(iii) $u_\epsilon$ is a viscosity supersolution of (\ref{viscosity solution}) in $\Omega''$, and $\lambda(A^{u_\epsilon})\in\bar{\Gamma}$ where $u$ is second order differentiable in $\Omega''$.	
\end{lemm}

\begin{proof}
	
	We will prove it for $\Omega=B_{2R}$, $\Omega'=B_R$, $\Omega''=B_{R/2}$; the general case can be proved in the same way. Let $x_0$, $x_1\in \bar{B}_R$. Obviously there exists $x_0^{*}$, such that $\displaystyle u_\epsilon(x_0)=u(x_0^{*})+\frac{1}{\epsilon}|x_0^{*}-x_0|^2$, moreover, $u_\epsilon(x_0)\leq u(x_0)$.
	So we obtain:
	\begin{align*}
	\frac{1}{\epsilon}|x_0^{*}-x_0|^2=u_\epsilon(x_0)-u(x_0^{*})\leq u(x_0)-u(x_0^{*}),
	\end{align*}
	hence
	\begin{align*}
	|x_0^{*}-x_0|^2\leq \epsilon\ \text{osc} u,
	\end{align*}
	and we also know $u(x_0^{*})-u(x_0)\leq u_\epsilon(x_0)-u(x_0)\leq 0$. Now by the continuity of $u$, we obtain (i).
	
	Take $x\in\bar{B}_R$, 
	\begin{align*}
	u_\epsilon(x_0)&\leq u(x)+\frac{1}{\epsilon}|x-x_0|^2\\
	&\leq u(x)+\frac{1}{\epsilon}|x-x_1|^2+\frac{2}{\epsilon}|x-x_1||x_1-x_0|+\frac{1}{\epsilon}|x_1-x_0|^2\\
	&\leq u(x)+\frac{1}{\epsilon}|x-x_1|^2+\frac{3}{\epsilon}\cdot 2R|x_1-x_0|.
	\end{align*}
	Take supremum over $x\in\bar{B}_R$, we obtain $\displaystyle	u_\epsilon(x_0)\leq u_\epsilon(x_1)+\frac{6R}{\epsilon}|x_1-x_0|$. Hence $\displaystyle|	u_\epsilon(x_0)-u_\epsilon(x_1)|\leq \frac{6R}{\epsilon}|x_1-x_0|$, so $u_\epsilon$ is locally Lipschitz.

	Define
	\begin{align*}
	P_0(x)=u(x_0^{*})+\frac{1}{\epsilon}|x-x_0^{*}|,
	\end{align*}
	so it has the property
	\begin{align*}
	&P_0(x_0)=u_\epsilon(x_0),\\
	&P_0(x)\geq u_\epsilon(x),\ x\in \bar B_{R}.
	\end{align*}
	That is, $P_0$ touches $u_\epsilon$ from the above at $x_0$ in $\bar{B}_R$. So for small $h$, $\Delta_h^2(u_\epsilon(x_0)-\frac{1}{\epsilon}|x_0|^2)\leq 0$,
	where
	\begin{align*}
	\Delta_h^2u(x_0):=\frac{u(x_0+h)+u(x_0-h)-2u(x_0)}{|h|^2}.
	\end{align*}
	This implies that $u_\epsilon(x)-\frac{1}{\epsilon}|x|^2$ is concave in $B_R$, so we have obtained (ii).
	
	Let $x_0\in B_{R/2}$, and let $P(x)$ be a paraboloid touches $u_\epsilon$ by below at $x_0$. Consider the paraboloid
	\begin{align*}
	Q(x)=P(x+x_0-x_0^{*})-\frac{1}{\epsilon}|x_0-x_0^{*}|^2.
	\end{align*}
	Since
	\begin{align*}|x_0^{*}-x_0|^2\leq \epsilon\ \text{osc} u,
	\end{align*}
	we can pick $\epsilon_0$ such that for $\forall\epsilon\leq\epsilon_0$, if $ x_0\in B_{R/2}$, then $x_0^{*}\in B_R$. Take any $x$ sufficient closed to $x_0^*$ so that $x+x_0-x_0^*\in B_R$.
	
	By the definition of $u_\epsilon$, 
	\begin{align*}
	u_\epsilon(x+x_0-x_0^*)\leq u(x)+\frac{1}{\epsilon}|x_0-x_0^*|^2.
	\end{align*}
	If $x$ is sufficiently close to $x_0^*$, then by the assumption on $P$, we obtain
	\begin{align*}
	u_\epsilon(x+x_0-x_0^*)\geq P(x+x_0-x_0^*).
	\end{align*}
	Therefore, we have
	\begin{align*}
	Q(x)=P(x+x_0-x_0^*)-\frac{1}{\epsilon}|x_0-x_0^*|^2\leq u(x).
	\end{align*}
	At $x_0^*$, we have, by definition,
	\begin{align*} Q(x_0^*)=P(x_0)-\frac{1}{\epsilon}|x_0-x_0^{*}|^2=u_\epsilon(x_0)-\frac{1}{\epsilon}|x_0-x_0^{*}|^2=u(x_0^*).
	\end{align*}
	This shows $Q$ is a paraboloid touching $u$ from below at $x_0^*$. Since $u$ is a viscosity supersolution, 
	\begin{align*}
	\lambda(A^{Q(x_0^*)})\in \bar{\Gamma}.
	\end{align*}
	However, $\nabla Q(x_0^*)=\nabla P(x_0)$, $\nabla^2 Q(x_0^*)=\nabla^2 P(x_0)$, so we obtain
	\begin{align*}
	\lambda(A^{P(x_0)})\in \bar{\Gamma}.
	\end{align*}
	This implies that $u_\epsilon$ is a viscosity supersolution. (iii) is proved.
\end{proof}

In the remaining part of this section, we establish the following asymptotic behavior of viscosity supersolution of (\ref{viscosity solution}).

\begin{prop}\label{Asymptotic behavior-1}
	Let $\Gamma=\Gamma_p$ for some $1<p\leq 2$, and let $u$ be a continuous viscosity supersolution of (\ref{viscosity solution}) in $\mathbb{R}^2\setminus B_{\frac{r_0}{2}}$ for some $r_0>0$.
	
	Then there exists $K_0>0$, such that
	\begin{align}\label{monotone}
	\inf_{\partial B_r}u(r)+4\ln r\text{ is monotonically nondecreasing for }r>K_0.
	\end{align}
	Consequently, $\displaystyle\liminf_{x\rightarrow \infty} \left(u(x)+4\ln |x|\right)>-\infty$.
\end{prop}

\begin{proof}

\medskip

Let 
\begin{align*}
v(r):=\inf_{|x|=r}u(x).
\end{align*}
It follows that $v$ is a continuous viscosity supersolution of (\ref{viscosity solution}) in $\mathbb{R}^2\setminus B_{r_0}$ and $v$ is radially symmetric. 

Define
\begin{align*}
\tilde{v}_{\epsilon,R}(x)=\inf_{y\in \bar{B}_{2R}\backslash B_{
\frac{3r_0}{4}}}\{v(y)+\frac{1}{\epsilon}|y-x|^2\},\quad \forall\ x\in  \bar{B}_{2R}\backslash B_{\frac{3r_0}{4}},
\end{align*}
for $R>r_0+2$, $0<\epsilon<\epsilon_0(R)$, where $\epsilon_0(R)=\epsilon_0(\Omega',\Omega'')>0$ is as defined in Lemma \ref{Jensen Approximation} with $\Omega'=B_{2R}\backslash \bar{B}_\frac{3r_0}{4}$, $\Omega''=B_R\backslash \bar{B}_{r_0}$. Clearly, $\tilde{v}_{\epsilon,R}(x)$ is radially symmetric, and we will use $\tilde{v}_{\epsilon,R}(r)$ to denote $\tilde{v}_{\epsilon,R}(x)$, for $|x|=r$.

Since $\tilde{v}_{\epsilon,R}(r) -\frac{1}{\epsilon}r^2$ is concave,  by the Rademacher's theorem and the Alexandroff's theorem, $\tilde{v}_{\epsilon,R}$ is differentiable almost everywhere, $\tilde{v}_{\epsilon,R}'(r) -\frac{2}{\epsilon}r$ is monotonically nonincreasing, and $\tilde{v}_{\epsilon,R}$ is second order differentiable almost everywhere. More precisely, there exists $E_{\epsilon,R}\subset (r_0,R)$ with $|E_{\epsilon,R}\cap (r_0,R)|=R-r_0$, satisfying: 

For $r\in E_{\epsilon,R}$,

(i)$\displaystyle\tilde{v}_{\epsilon,R}'(r):=\lim_{h\rightarrow 0} \frac{\tilde{v}_{\epsilon,R}(r+h)-\tilde{v}_{\epsilon,R}(r)}{h}$ is finite.

(ii)$\displaystyle \tilde{v}_{\epsilon,R}''(r):=\lim_{h\rightarrow 0} \frac{\tilde{v}_{\epsilon,R}'(r+h)-\tilde{v}_{\epsilon,R}'(r)}{h}$ is finite.

where $\tilde{v}_{\epsilon,R}'(\cdot)$ is the extension to $(r_0,R)$ of $\tilde{v}'_{\epsilon,R}$ defined on $E_{\epsilon,R}$, given by
\begin{align*}
\tilde{v}_{\epsilon,R}'(t):=\lim_{\delta\rightarrow 0}\frac{1}{2\delta}\int_{t-\delta<s<t+\delta}\tilde{v}_{\epsilon,R}'(s)ds \text{ for }\forall\ t\in (r_0,R).
\end{align*}

(iii)$\displaystyle\lim_{h\rightarrow 0} \frac{|\tilde{v}_{\epsilon,R}(r+h)-\tilde{v}_{\epsilon,R}(r)-\tilde{v}_{\epsilon,R}'(r)h-\frac{1}{2}\tilde{v}_{\epsilon,R}''(r)h^2|}{h^2}=0$.

In the following, when there is no ambiguity, we use $\tilde{v}$ to denote $\tilde{v}_{\epsilon,R}$.

Denote $\lambda(A^{\tilde{v}})=(\lambda_1,\lambda_2)$, then  
\begin{align*}
\lambda_1(r)=\frac{1}{e^{\tilde{v}}}\left( -\tilde{v}^{\prime\prime}+\frac{1}{4}(\tilde{v}^\prime)^2\right),\quad \lambda_2(r)=\frac{1}{e^{\tilde{v}}} \left(-\frac{\tilde{v}^\prime}{r}-\frac{1}{4}(\tilde{v}^\prime)^2\right),\quad \text{in }E_{\epsilon,R}.
\end{align*}

By Lemma \ref{Jensen Approximation} (ii)(iii), $\tilde{v}'-\frac{2}{\epsilon}r$ is monotonically nonincreasing, and $\lambda(A^{\tilde{v}})\in\bar\Gamma$ in $E_{\epsilon,R}$,

We distinguish into two cases.

Case 1: $\Gamma=\Gamma_2$.

By Lemma \ref{Jensen Approximation} (iii), $(\lambda_1,\lambda_2)\in\bar{\Gamma}_2$ in $E_{\epsilon,R}$, so
\begin{align*}
\lambda_2=\frac{1}{e^{\tilde{v}}} \left(-\frac{\tilde{v}^\prime}{r}-\frac{1}{4}(\tilde{v}^\prime)^2\right)=-\frac{\tilde{v}^\prime}{e^{\tilde{v}}} \left(\frac{1}{r}+\frac{1}{4}\tilde{v}^\prime\right)\geq 0,\quad \text{in }E_{\epsilon,R}.
\end{align*}

It follows that

\begin{align*}
(\tilde{v}+4\ln r)^\prime\geq  0,\quad \text{in }E_{\epsilon,R}.
\end{align*}

Since $\tilde{v}$ is locally Lipschitz, $\tilde{v}+4\ln r$ is monotonically nondecreasing in $(r_0,R)$.
Sending $\epsilon\rightarrow 0$, using Lemma \ref{Jensen Approximation} (i),  we obtain (\ref{monotone}).

Case 2: $\Gamma_2\subsetneqq\Gamma\subsetneqq\Gamma_1$.

\medskip

	Define
	
	\begin{align*}
	\tilde{E}_{\epsilon,R}:=\{r\in E_{\epsilon,R}: \tilde{v}^\prime(r)<-\frac{4}{r}\}.
	\end{align*}

	When there is no confusion, denote
	\begin{align*}
	E=E_{\epsilon,R},\quad \tilde{E}=\tilde{E}_{\epsilon,R}.
	\end{align*}
	It is clear that

	\begin{align}\label{lambda1}
	\lambda_2=\frac{1}{e^{\tilde{v}}} \left(-\frac{\tilde{v}^\prime}{r}-\frac{1}{4}(\tilde{v}^\prime)^2\right)=-\frac{(\tilde{v}^\prime)^2}{e^{\tilde{v}}} \frac{(\frac{1}{\tilde{v}^\prime}+\frac{r}{4})}{r}<0,\quad\text{on }\tilde{E}.
	\end{align}
	
	Since $(\lambda_1,\lambda_2)\in\Gamma$, we also have $\lambda_1>0$ on $\tilde{E}$, so
	\begin{align}\label{lambda2}
	\lambda_1=\frac{1}{e^{\tilde{v}}}\left( -\tilde{v}^{\prime\prime}+\frac{1}{4}(\tilde{v}^\prime)^2\right)=\frac{(\tilde{v}^\prime)^2}{e^{\tilde{v}}}\left(\frac{1}{\tilde{v}^\prime}+\frac{r}{4}\right)^\prime>0,\quad \text{on }\tilde{E}.
	\end{align}
	
		Denote
	\begin{align*}
	g(r)=\frac{1}{\tilde{v}^\prime(r)}+\frac{r}{4}.
	\end{align*}
	
	From (\ref{lambda1}) and (\ref{lambda2}), we have
	\begin{align*}
	g>0,\quad g^\prime>0,\quad \text{in }\tilde{E}.
	\end{align*}
	
	We start with a lemma:
	\begin{lemm}\label{maxinteval}
	For $R>r_0+2$, $0<\epsilon<\epsilon_0(R)$, if $a\in \tilde{E}_{\epsilon,R}$, then
	\begin{align*}
	(a,R)\cap \tilde{E}_{\epsilon,R}=(a,R)\cap E_{\epsilon,R}.
	\end{align*}
	Moreover, $\forall\ a<\alpha<\beta<R$, $\alpha,\beta\in E$, we have
	\begin{align}\label{gmonotone}
	0<\int_{\alpha}^{\beta}g'(r)dr\leq g(\beta)-g(\alpha).
	\end{align}
    \end{lemm}	
    \begin{proof}
	
	Let $\bar{r}\in \tilde{E}$. By Lemma \ref{Jensen Approximation} (ii), $\tilde{v}^\prime-\frac{2}{\epsilon}r$ is monotonically nonincreasing, so
	\begin{align}\label{decreasingderi}
		\lim_{r\rightarrow \bar{r}^+}\tilde{v}'(r)\leq \tilde{v}'(\bar{r})\leq\lim_{r\rightarrow \bar{r}^-}\tilde{v}'(r).
	\end{align}
	Therefore,
	\begin{align*}
	\lim_{r\rightarrow \bar{r}^+}(\tilde{v}'(r)+\frac{4}{r})\leq \tilde{v}'(\bar{r})+\frac{4}{\bar{r}}<0.
	\end{align*}
	Hence there exists $\delta>0$, such that $\tilde{v}'(r)+\frac{4}{r}<0$ for any $\bar{r}<r<\bar{r}+\delta$, i.e. $(\bar{r}, \bar{r}+\delta)\cap E\subset \tilde{E}$.
	
	Since $a\in \tilde{E}$, $b:=\sup\{c:(a,c)\cap E\subset \tilde{E}\}$ is well defined.

		For $a<\alpha<s<\beta<b$, $\alpha,s,\beta\in E$, since $\tilde{v}'(r)-\frac{2}{\epsilon}r$ is monotonically nonincreasing, we have
		\begin{align*}
		\tilde{v}'(s)-\tilde{v}'(s+\frac{1}{m})\geq -\frac{2}{\epsilon}\frac{1}{m}.
		\end{align*}
		
		By the definition of $\tilde E$, we have, using $-\tilde{v}'(r)\geq \frac{4}{r}$, for any $a<r<b$, $r\in E$,
		\begin{align*}
		0\leq\frac{1}{\tilde{v}^\prime(r+\frac{1}{m})\tilde{v}^\prime(r)}\leq \frac{r(r+\frac{1}{m})}{16}.
		\end{align*}
		
		Therefore,
		\begin{align*}
	    m\left(\frac{1}{\tilde{v}^\prime(s+\frac{1}{m})}-\frac{1}{\tilde{v}^\prime(s)}\right)&=\frac{m(\tilde{v}^\prime(s)-\tilde{v}^\prime(s+\frac{1}{m}))}{\tilde{v}^\prime(s+\frac{1}{m})\tilde{v}^\prime(s)}\geq - \frac{s(s+\frac{1}{m})}{8\epsilon}\geq-\frac{b^2}{8\epsilon}.
		\end{align*}
		
		So let
		\begin{align*}
		h_m(s)=m\left(\frac{1}{\tilde{v}^\prime(s+\frac{1}{m})}-\frac{1}{\tilde{v}^\prime(s)}\right)\geq-\frac{b^2}{8\epsilon},\ \alpha<s<\beta.
		\end{align*}
		
		We have
		\begin{align*}
		\lim_{m\rightarrow \infty}h_m(s)=\left(\frac{1}{\tilde{v}^\prime(s)}\right)^\prime,\quad \text{in }E.
		\end{align*}
		
		By Fatou's lemma,
		\begin{align*}
		\int_{\alpha}^{\beta}\left(\frac{1}{\tilde{v}^\prime(s)}\right)^\prime ds\leq& \liminf_{m\rightarrow \infty}\int_{\alpha}^{\beta} h_m(s)ds\\
		=& \liminf_{m\rightarrow\infty}(-m\int_{\alpha}^{\alpha+\frac{1}{m}}\frac{1}{\tilde{v}'(s)}ds+m\int_{\beta}^{\beta+\frac{1}{m}}\frac{1}{\tilde{v}'(s)}ds)\\
		=&\frac{1}{\tilde{v}^\prime(\beta)}-\frac{1}{\tilde{v}^\prime(\alpha)}.
		\end{align*}
		
		Thus, using $g'>0$ in $\tilde{E}$,  (\ref{gmonotone}) follows.
		
		Now we will prove $b=R$.

		By (\ref{gmonotone}), using also $\tilde{v}'<0$ on $\tilde{E}$, there exists $\mu>0$, such that
		\begin{align*}
		\tilde{v}'+\frac{4}{r}<-\mu<0\text{ on }(\frac{a+b}{2},b)\cap E.
		\end{align*}
		
		If $b\neq R$, then by inequality (\ref{decreasingderi}), there exists $\delta>0$, such that
		\begin{align*}
		\tilde{v}'+\frac{4}{r}<-\frac{\mu}{2}<0\text{ on }[b,b+\delta)\cap E.
		\end{align*}

		This violates the definition of $b$.
		
		Thus $b=R$. Lemma \ref{maxinteval} is now proved.
    \end{proof}

	\medskip

	For given $R>r_0+2$ and $0<\epsilon<\epsilon_0(R)$, define $a_0^{\epsilon,R}$ to be:
	\begin{align*}
	a_0^{\epsilon,R}=\inf\{a:a\in \tilde{E}_{\epsilon,R},\text{ if }\tilde{E}_{\epsilon,R}\neq\varnothing\}.
	\end{align*}
	If $\tilde{E}_{\epsilon,R}=\varnothing$ then define $a_0^{\epsilon,R}=R$.
	
	If $\tilde{E}_{\epsilon,R}\neq\varnothing$, then $r_0\leq a_0^{\epsilon,R}<R$, and, by Lemma \ref{maxinteval}, $(a_0^{\epsilon,R},R)\cap \tilde{E}_{\epsilon,R}=(a_0^{\epsilon,R},R)\cap E_{\epsilon,R}$.

    Since $\Gamma_2\subsetneqq\Gamma\subsetneqq \Gamma_1$, there exists a unique constant $0<p<1$ such that $(1,-p)\in \partial \Gamma$.

    Note that
    \begin{align}\label{lambdaineq}
    \lambda_2>-p\lambda_1\quad\text{in }(a_0^{\epsilon,R},R)\cap \tilde{E}_{\epsilon,R}.
    \end{align}
	
	We will prove,
	\begin{align}\label{kmonotone}
	\alpha^{-\frac{1}{p}}g(\alpha)< {\beta}^{-\frac{1}{p}}g(\beta), \quad \text{for }a_0^{\epsilon,R}<\alpha<\beta<R.
	\end{align}

	Let $k(r)=r^{-\frac{1}{p}}g(r)$, we have, using (\ref{lambda1}), (\ref{lambda2}) and (\ref{lambdaineq}),
	\begin{align*}
	k'(r)>0,\text{ in }(a_0^{\epsilon,R},R)\cap \tilde{E}_{\epsilon,R}.
	\end{align*}

	Set $\tilde{h}_m(s)=m(k(s+\frac{1}{m})-k(s))$, $a_0^{\epsilon,R}<\alpha<s<\beta<R$, $\alpha,s,\beta\in \tilde{E}_{\epsilon,R}$, then using Lemma \ref{maxinteval}, we have
	\begin{align*}
	\tilde{h}_m(s)&=m((s+\frac{1}{m})^{-\frac{1}{p}}g(s+\frac{1}{m})-s^{-\frac{1}{p}}g(s))\geq m((s+\frac{1}{m})^{-\frac{1}{p}}-s^{-\frac{1}{p}})g(s)\geq -C
	\end{align*}
	and
	\begin{align*}
	\tilde{h}_m(s)\rightarrow k'(s), \text{ as }m\rightarrow\infty,\text{ for }a_0^{\epsilon,R}<\alpha<s<\beta<R,\ s\in \tilde{E}_{\epsilon,R}.
	\end{align*}

	By Fatou's lemma, we have
	\begin{align*}
	k(\beta)-k(\alpha)&=\liminf_{m\rightarrow \infty}\big(m \int_{\beta}^{\beta+\frac{1}{m}} k(s)ds-m\int_{\alpha}^{\alpha+\frac{1}{m}} k(s)ds\big)\\
	&=\liminf_{m\rightarrow \infty} \int_\alpha^\beta \tilde{h}_m(s)ds \geq \int_\alpha^\beta k'(s)ds>0,\ \alpha,\beta\in \tilde{E}_{\epsilon,R}.
	\end{align*}
	
	Hence (\ref{kmonotone}) follows.
\\

By Lemma \ref{Jensen Approximation} (i), for any $R>r_0+2$, there exists $\epsilon'(R)\rightarrow 0^+$ as $R\rightarrow\infty$, such that for $\forall\ 0<\epsilon<\epsilon'(R)$,
\begin{align}\label{estimatev}
|\tilde{v}_{\epsilon,R}(x)-v(x)|\leq e^{-R},\quad \text{for }\forall\ r_0\leq x\leq R.
\end{align}
We will have two cases:\\
Case 1: there exists $R_i\rightarrow\infty$, and $0<\epsilon_i<\epsilon'(R_i)$, such that $a_0^{\epsilon_i,R_i}\rightarrow \infty$.
\\
Case 2: there exists $R_i\rightarrow\infty$, and $0<\epsilon_i<\epsilon'(R_i)$, such that $a_0^{\epsilon_i,R_i}\rightarrow K_0<\infty$.
\\
For Case 1:

For any $r_0<\alpha<\beta$, $\beta<a_0^{\epsilon_i,R_i}$ for large $i$. So
\begin{align*}
\frac{d}{dr}(\tilde{v}_{\epsilon_i,R_i}(r)+4\ln r)=\tilde{v}'_{\epsilon_i,R_i}(r)+\frac{4}{r}\geq 0,\quad\forall\ \alpha\leq r\leq \beta.
\end{align*}
Integrate on $[\alpha,\beta]$, by Lemma \ref{Jensen Approximation} (ii),
\begin{align*}
\tilde{v}_{\epsilon_i,R_i}(\beta)+4\ln \beta\geq\tilde{v}_{\epsilon_i,R_i}(\alpha)+4\ln \alpha.
\end{align*}
Send $i$ to $\infty$, using (\ref{estimatev}),
\begin{align*}
v(\beta)+4\ln \beta\geq v(\alpha)+4\ln \alpha.
\end{align*}
\ 
\\
For Case 2:

For any $K_0<\alpha<\beta<\infty$, $[\alpha,\beta]\subset (a_0^{\epsilon_i,R_i},\frac{R_i}{2})$ for large $i$.
\\
Using (\ref{kmonotone}),  we have
\begin{align*}
\frac{1}{4}\cdot\frac{R_i}{2}>\frac{1}{\tilde{v}'_{\epsilon_i,R_i}(\frac{R_i}{2})}+\frac{1}{4}\cdot\frac{R_i}{2}>\frac{g(x)}{x^{\frac{1}{p}}}(\frac{R_i}{2})^{\frac{1}{p}},\quad \text{for }x\in [\alpha,\beta].
\end{align*}
Since $0<p<1$, it follows that
\begin{align*}
\frac{1}{\tilde{v}'_{\epsilon_i,R_i}(x)}+\frac{x}{4}<C R_i^{1-1/p}=o(1),\quad \text{ uniform for }x\in [\alpha,\beta].
\end{align*}
Therefore,
\begin{align*}
\frac{d}{dx}(\tilde{v}_{\epsilon_i,R_i}(x)+4\ln x)=\tilde{v}'_{\epsilon_i,R_i}(x)+\frac{4}{x}>o(1),\text{ uniform for }x\in [\alpha,\beta].
\end{align*}
By Lemma \ref{Jensen Approximation} (ii), integrating on $x$, we obtain
\begin{align*}
\tilde{v}_{\epsilon_i,R_i}(\beta)+4\ln{\beta}>\tilde{v}_{\epsilon_i,R_i}(\alpha)+4\ln \alpha+o(1).
\end{align*}
Sending $i\rightarrow\infty$, using (\ref{estimatev}) and $R_i\rightarrow\infty$, we obtain
\begin{align*}
v(\beta)+4\ln{\beta}\geq v(\alpha)+4\ln \alpha.
\end{align*}
\\
Proposition \ref{Asymptotic behavior-1} is now proved.

\end{proof}

\section{Liouville Theorem}\label{proofLiou}

In this section, we prove Theorem \ref{nondege}. Given the asymptotic behavior established in Section \ref{viscosity}, we can handle the possible singularity of $u$  at infinity either by following the proof in \cite[Theorem 1.3]{LL2}, or by invoking a general result in \cite{CLN}. Here we give the latter.

Recall:

\noindent {\bf Theorem A.} \cite[Theorem 1.1]{CLN}\ \label{touching}
For $n\geq 1$, let $\Omega$ be a domain in $\mathbb{R}^n$ containing the origin, and let $F\in C^1(\Omega,\mathbb{R},\mathbb{R}^n,\mathcal{S}^{n\times n})$ satisfy
\begin{align*}
-\frac{\partial F}{\partial M_{ij}}(x,s,p,M)>0,\quad \forall (x,s,p,M)\in\Omega\times\mathbb{R}\times\mathbb{R}^n\times\mathcal{S}^{n\times n}.
\end{align*}
Assume that $u\in C^2(\Omega\backslash\{0\})$ satisfies
\begin{align}\label{superhar}
\text{For any }V\in\mathbb{R}^n, w(x):=u(x)+V\cdot x\text{ satisfies }\inf_{B_r\backslash\{0\}}w=\min_{\partial B_r}w,\ \forall \ 0<r<\bar{r},
\end{align}
for some $\bar r>0$, $v\in C^2(\Omega)$, and
\begin{align*}
u&>v\quad \text{in }\Omega\backslash\{0\},\\
F(x,u,\nabla u,\nabla^2 u)&\geq F(x,v,\nabla v,\nabla^2 v)\quad \text{in }\Omega\backslash\{0\}.
\end{align*}
Then
\begin{align*}
\liminf_{x\rightarrow 0}(u-v)(x)>0.
\end{align*}

\begin{rema}
	In $n\geq 2$, a superharmonic function $u\in C^0(B_1\backslash\{0\})$ satisfying $\displaystyle\inf_{B_1\backslash\{0\}}u>-\infty$ has the above property (\ref{superhar}).
\end{rema}

\medskip
\noindent
Proof of Theorem \ref{nondege}:

We use the method of moving spheres. The proof is similar to that of \cite[Theorem 1.3]{LL2}, see also the proofs of \cite[Theorem 1.1]{LZ} and \cite[Theorem 1.1]{LZhu}.

To begin with, we define
\begin{align*}
u_{x,\lambda}(y)=u(x+\frac{\lambda^2(y-x)}{|y-x|^2})-4\ln\frac{|y-x|}{\lambda}.
\end{align*}
Denote $F(A^u)=f(\lambda(A^u))$.
Notice that if $F(A^u)=1$, then $F(A^{u,\lambda})=1$, where we have used the invariance property of $A^u$ and $F$.

For the sake of convenience, denote
\begin{align*}
u_{\lambda}(y)=u(\frac{\lambda^2y}{|y|^2})-4\ln\frac{|y|}{\lambda}.
\end{align*}

\begin{lemm}\label{startmove}
For every $x\in \mathbb{R}^2$, there exists $\lambda_0(x)>0$ such that $u_{x,\lambda}(y)\leq u(y)$, for all $0<\lambda<\lambda_0(x)$ and $|y-x|\geq \lambda$.
\end{lemm}
\begin{proof}
Consider the function $u(r,\theta)+2\ln r$, we have
\begin{align*}
\frac{\partial}{\partial r}\left(u(r,\theta)+2\ln r\right)=\frac{\partial u}{\partial r}+\frac{2}{r}.
\end{align*}

Thus there exists $r_0>0$, such that for all $0<r<r_0$, we have
\begin{align*}
\frac{\partial}{\partial r}\left(u(r,\theta)+2\ln r\right)>0.
\end{align*}

It follows that for $0<r_1\leq  r_2<r_0$, we have
\begin{align*}
u(r_1,\theta)+2\ln r_1\leq u(r_2,\theta)+2\ln r_2.
\end{align*}

Choose $r_2=|y|$, $r_1=\dfrac{\lambda^2}{|y|}$, then for all $0<\lambda\leq |y|<r_0$, we have
\begin{align*}
u(\frac{\lambda^2}{|y|},\theta)+2\ln\frac{\lambda^2}{|y|}\leq u(|y|,\theta)+2\ln |y|,
\end{align*}

i.e.
\begin{align*}
u(\frac{\lambda^2}{|y|},\theta)-4\ln\frac{|y|}{\lambda}\leq u(|y|,\theta).
\end{align*}

It follows that for all $0<\lambda\leq |y|<r_0$, we have
\begin{align*}
u_{0,\lambda}(y)\leq u(y).
\end{align*}

Now let us consider $|y|\geq r_0$. By Proposition \ref{Asymptotic behavior-1}, we have
\begin{align*}
\liminf_{|y|\rightarrow \infty} (u(y)+4\ln|y|)>-\infty.
\end{align*}

Thus there exists a constant $a$, such that for all $|y|\geq r_0$, we have
\begin{align*}
u(y)+4\ln|y|\geq a.
\end{align*}

Then for $|y|\geq r_0\geq \lambda$, we have
\begin{align*}
&u(y)-u_{0,\lambda}(y)\\
=& u(y)-u(\frac{\lambda^2 y}{|y|^2})+4\ln\frac{|y|}{\lambda}\\
=&u(y)+4\ln |y|-u(\frac{\lambda^2 y}{|y|^2})-4\ln\lambda\\
\geq & a -u(\frac{\lambda^2 y}{|y|^2})-4\ln\lambda.
\end{align*}

It follows that there exists $0<\lambda_0\leq r_0$, such that for all $|y|\geq r_0$, $0<\lambda\leq \lambda_0$, we have
\begin{align*}
u(y)\geq u_{0,\lambda}(y).
\end{align*}

The lemma is now proved.

\end{proof}

\bigskip

By Lemma \ref{startmove}, we can define for $x\in\mathbb{R}^2$,
\begin{align*}
\bar{\lambda}(x):=\sup\{\mu:u_{x,\lambda}(y)\leq u(y), |y-x|\geq \lambda,0<\lambda<\mu\}\in (0,+\infty].
\end{align*}

By Proposition \ref{Asymptotic behavior-1},
\begin{align*}
\alpha:=\liminf_{|x|\rightarrow \infty}(u(x)+4\ln|x|)\in (-\infty,+\infty].
\end{align*}

We have two cases:

Case 1: $\alpha=+\infty$.

We will prove that this case does not occur. We first prove that 
\begin{align}\label{lambdainfty}
\bar{\lambda}(x)=+\infty\text{ for all }x\in \mathbb{R}^2.
\end{align}
Once (\ref{lambdainfty}) is proved, by Lemma \ref{Calculus lemma}, we obtain that $u$ must be constant and therefore $A^u=0$. This violates the condition $\lambda(A^u)\in\Gamma$. Hence Case 1 does not occur.

Without loss of generality, we only need to prove (\ref{lambdainfty}) for $x=0$. Suppose $\bar{\lambda}:=\bar{\lambda}(0)<\infty$. For each $\lambda>0$ fixed, we have 
\begin{align*}
u(y)-u_{0,\lambda}(y)=u(y)+4\ln |y|-u(\frac{\lambda^2 y}{|y|^2})-4\ln\lambda,\ y\in\mathbb{R}^2\backslash\{0\}.
\end{align*}

Since $\alpha=+\infty$, we have, for every $\lambda>0$, $u(y)-u_{0,\lambda}(y)\rightarrow +\infty$ as $|y|\rightarrow\infty$. It follows that there exists $M>0$, such that
\begin{align}\label{ulambdalesu1}
u_{0,\lambda}(y)<u(y),\quad 0\leq \lambda\leq \bar{\lambda}+1,\quad |y|\geq M.
\end{align}
Since $F(A^u)=1$ is M\"{o}bius invariant, we have $F(A^{u_\lambda})=1$. Therefore, by the condition $\partial_{\lambda_i}f>0$ in $\Gamma$, there exists a linear second order elliptic operator $L$, such that 
\begin{align*}
&L(u-u_{\bar{\lambda}})=0,\quad \mathbb{R}^2\setminus B_{\bar{\lambda}},\\
&u-u_{\bar{\lambda}}=0,\quad \partial B_{\bar{\lambda}}.
\end{align*}

By the maximum principle and the Hopf Lemma, we have
\begin{align}
&u-u_{\bar{\lambda}}>0,\quad \mathbb{R}^2\setminus B_{\bar{\lambda}},\label{ugequlambda}\\
&\frac{\partial}{\partial r}(u-u_{\bar{\lambda}})>0,\quad \partial B_{\bar{\lambda}}.\label{hopf}
\end{align}

By the compactness of $\partial B_{\bar{\lambda}}$ and (\ref{hopf}), we have
\begin{align*}
\frac{\partial}{\partial r}(u-u_{\bar{\lambda}})|_{\partial B_{\bar{\lambda}}}\geq b>0.
\end{align*}

By the continuity of $\nabla u$, there exists $\delta>0$, such that
\begin{align*}
\frac{\partial}{\partial r}(u-u_{\lambda})\geq \frac{b}{2},\quad \bar{\lambda}\leq \lambda\leq \bar{\lambda}+\delta,\quad \lambda\leq |y|\leq \lambda+\delta.
\end{align*}

Since $u=u_\lambda$ on $\partial B_\lambda$, we have
\begin{align}\label{case1ineq1}
u_{\lambda}\leq u,\quad \bar{\lambda}\leq \lambda\leq \bar{\lambda}+\delta,\quad \lambda\leq |y|\leq \lambda+\delta.
\end{align}

On the other hand, (\ref{ugequlambda}) implies
\begin{align*}
u_{\bar{\lambda}}<u,\quad \bar{\lambda}+\delta\leq y\leq M.
\end{align*}

By the continuity of $u$, there exists $0<\epsilon<\delta$, such that
\begin{align}\label{case1ineq2}
u_{\lambda}\leq u,\quad \bar{\lambda}\leq \lambda\leq \bar{\lambda}+\epsilon,\quad \bar{\lambda}+\delta\leq y\leq M.
\end{align}

By (\ref{ulambdalesu1}), (\ref{case1ineq1}) and (\ref{case1ineq2}), we have proved
\begin{align*}
u_{\lambda}\leq u, \quad \bar{\lambda}\leq \lambda\leq \bar{\lambda}+\epsilon,\quad |y|\geq \lambda.
\end{align*}

This violates the definition of $\bar{\lambda}$, and (\ref{lambdainfty}) is proved.

\bigskip

Case 2: $\alpha\in \mathbb{R}$.

We first claim that $\bar\lambda(x)<\infty$, $\forall\ x\in\mathbb{R}^2$. 

In fact, for all $|y-x|\geq \lambda$, $0<\lambda<\bar{\lambda}(x)$, we have
\begin{align*}
u_{x,\lambda}(y)\leq u(y),
\end{align*}

i.e.
\begin{align*}
u(x+\frac{\lambda^2(y-x)}{|y-x|^2})-4\ln\frac{|y-x|}{\lambda}\leq u(y).
\end{align*}
Fix $x\in\mathbb{R}^2,\lambda>0$, let $y\rightarrow\infty$, and then let $\lambda\rightarrow\bar{\lambda}(x)$, it follows that
\begin{align*}
\alpha\geq u(x)+4\ln\bar{\lambda}(x), \text{ for all }x\in\mathbb{R}^2.
\end{align*}

Therefore, $\bar\lambda(x)<\infty$, $\forall\ x\in\mathbb{R}^2$. The claim is proved.

\medskip

Next, we will prove 
\begin{align}\label{uequlambda}
u_{x,\bar{\lambda}(x)}\equiv u\text{ on }\mathbb{R}^2\backslash\{x\},\ \forall\ x\in\mathbb{R}^2.
\end{align}
Once (\ref{uequlambda}) is proved, Theorem \ref{nondege} follows from Lemma \ref{Calculuslemma2}

Without loss of generality, we only need to prove (\ref{uequlambda}) for $x=0$. We still denote $\bar\lambda=\bar\lambda(0)$. It suffices to prove
\begin{align}\label{stopmov}
u_{\bar\lambda}\equiv u,\quad\text{on }\mathbb{R}^2\setminus B_{\bar{\lambda}}.
\end{align}

From the definition of $\bar\lambda$,
\begin{align*}
u_{\bar\lambda}\leq u,\quad\text{on }\mathbb{R}^2\setminus B_{\bar{\lambda}}.
\end{align*}

Suppose on the contrary that (\ref{stopmov}) does not hold, then by the strong maximum principle,
\begin{align*}
u_{\bar\lambda}< u,\quad\text{on }\mathbb{R}^2\setminus B_{\bar{\lambda}}.
\end{align*}

This is equivalent to
\begin{align*}
u_{\bar\lambda}> u,\quad\text{on }B_{\bar{\lambda}}\backslash\{0\}.
\end{align*}

By Theorem A, we know
\begin{align*}
\liminf_{|y|\rightarrow0}(u_{\bar{\lambda}}(y)-u(y))>0,
\end{align*}

namely,
\begin{align*}
\liminf_{|y|\rightarrow\infty}(u(y)-u_{\bar{\lambda}}(y))>0.
\end{align*}

Hence there exists $\epsilon_0>0$ and $M>0$, such that
\begin{align*}
u(y)\geq u_{\bar{\lambda}}(y)+\epsilon,\ \forall\ |y|>M
\end{align*}

It follows that there exists $\delta>0$, such that
\begin{align*}
u(y)\geq u_{\lambda}(y),\quad 0\leq \lambda\leq \bar{\lambda}+\delta,\quad \forall\ |y|\geq M.
\end{align*}

Similar to Case 1, by using the Hopf Lemma, there exists $\epsilon>0$, such that
\begin{align*}
u(y)\geq u_{\lambda}(y),\quad 0\leq \lambda\leq \bar{\lambda}+\epsilon,\quad |y|\geq \lambda.
\end{align*}
This violates the definition of $\bar{\lambda}$. 

Therefore, we have proved (\ref{uequlambda}).

\section{Conformal Invariance and M\"{o}bius Invariance}\label{classi}
We first prove Proposition \ref{Mobiusinv}:
\begin{proof}
	\ 
	Let $x\in\mathbb{R}^2$, $s\in\mathbb{R}$, $p\in\mathbb{R}^2$, $M\in \mathcal S^{2\times 2}$, and $O\in O(2)$. Let $u$ be a smooth function satisfying $u(x)=s$, $\nabla u(x)=p$, and $\nabla^2 u(x)=M$. Consider $\psi(z)=z+x$. Now evaluate (\ref{Mobiusfor}) at the origin, we have
	\begin{align*}
	H(0,s,p,M)=H(x,s,p,M).
	\end{align*}
	So $H$ is independent of $x$. In the following, we denote $H(s,p,M)=H(0,s,p,M)$.
	
	Now define another smooth function $u$ satisfying $u(0)=s$, $\nabla u(0)=0$, and $\nabla^2 u(0)=M$. Consider $\psi(z)=Oz$, then evaluating (\ref{Mobiusfor}) at the origin gives:
	\begin{align}\label{ortho}
	H(s,0,O^{-1}MO)=H(s,0,M).
	\end{align}
	
	Next define a function $u$ satisfying $u(0)=s$, $\nabla u(0)=p$, and $\nabla^2 u(0)=M$. Consider $\psi(z)=e^{-s/2}z$, so $u_\psi(z)=u(e^{-s/2}z)-s$. Evaluating (\ref{Mobiusfor}) at the origin gives:
	\begin{align}\label{dia}
	H(0,e^{-s/2}p,e^{-s}M)=H(s,p,M).
	\end{align}
	
	Finally, define $\displaystyle\alpha=\frac{-4p}{|p|^2}\in \mathbb{R}^2$, and pick a smooth function $u$ satisfying $u(\alpha)=s$, $\nabla u(\alpha)=p$, and $\nabla^2 u(\alpha)=M$. Consider the M\"{o}bius transformation $\displaystyle\psi(z)=\frac{16z}{|p|^2|z|^2}$. By direct computation, $\displaystyle u_\psi(z)=u(\frac{16z}{|p|^2|z|^2})-4\ln |z|+4\ln 4-4\ln |p|$, and  $\displaystyle\psi^{-1}(\alpha)=\frac{-4p}{|p|^2}$.
	
	We can easily check that:
	\begin{align*}
	u_\psi\circ\psi^{-1}(\alpha)=u(\alpha)=s
	\end{align*}
	and
	\begin{align*}
	\nabla u_\psi\circ \psi^{-1}(\alpha)=0.
	\end{align*}
	So
	\begin{align*}
	A^{u_\psi}\circ\psi^{-1}(\alpha)=-\frac{\nabla^2 u_\psi}{e^{u_\psi}}\circ\psi^{-1}(\alpha)=-\frac{\nabla^2 u_\psi\circ\psi^{-1}(\alpha)}{e^s}.
	\end{align*}
	By a direct computation, we can check the property that for any $y\in\mathbb{R}^2$ and $\psi$ M\"{o}bius,
	\begin{align*}
	A^{u_\psi}\circ\psi^{-1}(y)\sim A^u(y),
	\end{align*}
	where the notation $A\sim B$ means $A$ and $B$ are orthogonally similar to each other.
	Hence
	\begin{align*}
	-\frac{\nabla^2 u_\psi\circ\psi^{-1}(\alpha)}{e^s}\sim A^u(\alpha),
	\end{align*}
	i.e.
	\begin{align}\label{simrela}
	\nabla^2 u_\psi\circ\psi^{-1}(\alpha)\sim -e^s A^u(\alpha).
	\end{align}	
	Evaluate (\ref{Mobiusfor}) at $\psi^{-1}(\alpha)$, using also (\ref{ortho}), (\ref{dia}) and (\ref{simrela}),
	\begin{align*}
	H(s,p,M)&=H(s,0,\nabla^2 u_\psi\circ \psi^{-1}(\alpha))\\
	&=H(s,0,-e^s A^u(\alpha))\\
	&=H(0,0,-A^u(\alpha)).
	\end{align*}
	The proposition follows.
\end{proof}
\ 
Then we prove Proposition \ref{conformalinv}:
\begin{proof}
	\ 
	
	It is easy to see that we still have
	\begin{align*}
	H(\cdot,u,\nabla u,\nabla^2 u)=F(A^u),
	\end{align*}
	where $F$ is invariant under orthogonal conjugation. However, we observe that for general meromorphic functions, the relation $A^{u_\psi}\sim A^u\circ\psi$ does not hold. Now we are going to prove that $F$ must be a function of the trace.
	
	Set $\psi(z)=iz^2$, i.e. $\psi(x,y)=(2xy,-x^2+y^2)$. Set $u(x,y)=ax^2$. So
	\begin{align*}
	A^u\circ\psi=\frac{-1}{e^{4ax^2y^2}}\begin{pmatrix}
	2a-4a^2x^2y^2 & 0 \\
	0 & 4a^2x^2y^2
	\end{pmatrix}.
	\end{align*}
	Evaluate $(\ref{1})$ at $(0,y)$, we have
	\begin{align*}
	A^u\circ\psi(0,y)=\begin{pmatrix}
	-2a & 0 \\
	0 & 0
	\end{pmatrix}.
	\end{align*}
	Similarly, by a direct computation,
	\begin{align*}
	u_\psi(x,y)=4ax^2y^2+\ln 4+\ln (x^2+y^2),\\
	A^{u_\psi}(0,y)=\begin{pmatrix}
	-2a-\frac{3}{4y^4} & 0 \\
	0 & \frac{3}{4y^4}
	\end{pmatrix}.
	\end{align*}
	Take appropriate $a$ and $y$, the relation $F(A^u)\circ\psi=F(A^{u_\psi})$ implies that
	\begin{align*}
	F(\begin{pmatrix}
	\lambda_1 & 0 \\
	0 & \lambda_2
	\end{pmatrix})=F(\begin{pmatrix}
	\lambda_1+\lambda_2 & 0 \\
	0 & 0
	\end{pmatrix}),
	\end{align*}
	for any $\{\lambda_1,\lambda_2\}$ satisfying that at least one of them is positive. Negative case comes from evaluating $(\ref{1})$ at $(x,0)$. Hence $F$ must be a function of the trace, i.e. $F(A^u)=g(-e^{-u}\Delta u)$. Moreover, it is straightforward to check that $g(-e^{-u}\Delta u)$ is conformally invariant.
\end{proof}

\appendix
\section{Two Calculus Lemmas}\label{appendix}

We now state two calculus lemmas for the reader's convenience.
\begin{lemm}\label{Calculus lemma}
Let $u\in C^1(\mathbb{R}^2)$ satisfy 
\begin{align*}
u(x+\frac{\lambda^2(y-x)}{|y-x|^2})-4\ln\frac{|y-x|}{\lambda}\leq u(y),  \quad \textit{for all} \quad  \lambda>0, \quad x\in\mathbb{R}^2, \quad |y-x|\geq\lambda.
\end{align*}
Then $u$ must be constant.
\end{lemm}

\begin{proof}
Let $f=e^u$, we have
\begin{align*}
(\frac{\lambda}{|y-x|})^4f(x+\frac{\lambda^2(y-x)}{|y-x|^2})\leq f(y),  \quad \textit{for all} \quad  \lambda>0, \quad x\in\mathbb{R}^2, \quad |y-x|\geq\lambda.
\end{align*}
By \cite[Lemma 11.1]{LZ} (see also \cite[Lemma 3.3]{LZhu}), we conclude that $f$ is a constant, hence $u$ is a constant.
\end{proof}

\begin{lemm}\label{Calculuslemma2}
	Let $u\in C^1(\mathbb{R}^2)$. Suppose that for every $x\in\mathbb{R}^2$, there exists $\lambda(x)>0$ such that
	\begin{align*}
	u_{x,\lambda(x)}(y)=u(y),\quad y\in\mathbb{R}^2\backslash\{x\}.
	\end{align*}
	Then for some $a>0$, $b>0$, $\bar{x}\in\mathbb{R}^2$,
	\begin{align*}
	u(x)=2\ln\frac{8a}{8|x-\bar{x}|^2+b}.
	\end{align*}
\end{lemm}
\begin{proof}
	Let $f=e^u$, we have, for every $x\in\mathbb{R}^2$, there exists $\lambda(x)>0$ such that
	\begin{align*}
	(\frac{\lambda}{|y-x|})^4f(x+\frac{\lambda^2(y-x)}{|y-x|^2})=f(y),\quad y\in\mathbb{R}^2\backslash\{x\}.
	\end{align*}
	By \cite[Lemma 11.1]{LZ} (see also \cite[Lemma 3.7]{LZhu}),
	\begin{align*}
	f(x)=\pm(\frac{a}{d+|x-\bar x|^2})^2.
	\end{align*}
	Lemma \ref{Calculuslemma2} follows.
\end{proof}


\begin{thebibliography}{60}
	
	
	\bibitem{AbEs}
	D. P. Abanto and J. M. Espinar,
	{\em Escobar type theorems for elliptic fully nonlinear degenerate equations}, 
	Amer. J. Math. 141 (2019), no. 5, 1179-1216.
	
	
	\bibitem{BCE}
	E. Barbosa, M. P. Cavalcante and J. M. Espinar,
	{\em Min‐Oo Conjecture for Fully Nonlinear Conformally Invariant Equations},  
	Comm. Pure Appl. Math. 72 (2019), no. 11, 2259-2281.
	
	\bibitem{BoSheng}
	L. Bo and W. Sheng,
	{\em Some Rigidity Properties for Manifolds with Constant $k$-Curvature of Modified Schouten Tensor}, 
	J. Geom. Anal. 29 (2019), no. 3, 2862-2887.
	
	\bibitem{CC}
	L. Caffarelli and X. Cabr\'e, 
	{\em Fully nonlinear elliptic equations}, 
	American Mathematical Society Colloquium Publications, 43. American Mathematical Society, Providence, RI, 1995.
	
	\bibitem{CGS}
	L. Caffarelli, B. Gidas and J. Spruck,
	{\em Asymptotic symmetry and local behavior of semilinear elliptic equations with critical Sobolev growth}, 
	Comm. Pure Appl. Math. 42 (1989), no. 3, 271–297.
	
	\bibitem{CLN}
	L. Caffarelli, Y.Y. Li and L. Nirenberg,
	{\em Some remarks on singular solutions of nonlinear elliptic equations. I}, 
	J. Fixed Point Theory Appl. 5 (2009), no. 2, 353–395. 
	
	\bibitem{Case}
	J. S. Case,
	{\em The weighted $\sigma_k$-curvature of a smooth metric measure space}, 
	Pacific J. Math. 299 (2019), no. 2, 339–399.
	
	\bibitem{CaseW1}
	J. S. Case and Y. Wang,
	{\em Boundary operators associated to the $\sigma_k$-curvature}, 
	Adv. Math. 337 (2018), 83–106.
	
	\bibitem{CaseW2}
	J. S. Case and Y. Wang,
	{\em On a fully nonlinear sharp Sobolev trace inequality}, 
	arXiv:1910.14232.
	
	\bibitem{CLiu}
	K. C. Chang and J. Q. Liu,
	{\em On Nirenberg’s problem}, 
	Internat. J. Math. 4 (1993), no. 1, 35–58.
	
	
	\bibitem{CGY1}
	S.-Y. A. Chang, M. J. Gursky and P. Yang,
	{\em An equation of Monge-Ampere type in conformal geometry, and four-manifolds of positive Ricci curvature},  
	Ann. of Math. (2) 155 (2002), no. 3, 709–787.
	
	\bibitem{CGY2}
	S.-Y. A. Chang, M. J. Gursky and P. Yang,
	{\em An a priori estimate for a fully nonlinear equation on four-manifolds}, 
	J. Anal. Math. 87 (2002), 151–186. 
	
	
	\bibitem{CHY}
	S.-Y. A. Chang, Z.-C. Han and P. Yang,
	{\em  On the prescribing $\sigma_2$ curvature equation on ${\mathbb S^4} $}, 
	Calc. Var. Partial Differential Equations 40 (2011), no. 3-4, 539–565.
	
	\bibitem{CY}
	S.-Y. A. Chang and P. Yang,
	{\em  Prescribing Gaussian curvature on $\mathbb{S}^2$}, 
	Acta Math. 159 (1987), no. 3-4, 215–259.
	
	\bibitem{CL}
	W. Chen and C. Li, 
	{\em Classification of solutions of some nonlinear elliptic equations}, 
	Duke Math. J. 63 (1991), no. 3, 615–622.
	
	\bibitem{ChouW}
	K. S. Chou and T. Y.-H. Wan,
	{\em Asymptotic radial symmetry for solutions of $\Delta u+ e^ u= 0$ in a punctured disc}, 
	Pacific J. Math. 163 (1994), no. 2, 269–276.
	
	\bibitem{FangW}
	H. Fang and W. Wei,
	{\em $\sigma_2$ Yamabe problem on conic 4-sphere}, 
	Calc. Var. Partial Differential Equations 58 (2019), no. 4, Paper No. 119, 19 pp.
	
	\bibitem{FangW2}
	H. Fang and W. Wei,
	{\em A $\sigma_{2}$ Penrose inequality for conformal asymptotically hyperbolic 4-discs}, 
	arXiv:2003.02875.
	
	\bibitem{GNN}
	B. Gidas, W. Ni and L. Nirenberg,
	{\em Symmetry and related properties via the maximum principle}, 
	Comm. Math. Phys. 68 (1979), no. 3, 209–243.
	
	
	\bibitem{GeW}
	Y. Ge and G. Wang,
	{\em On a fully nonlinear Yamabe problem}, 
	Ann. Sci. \'Ecole Norm. Sup. (4) 39 (2006), no. 4, 569–598.
	
	\bibitem{GLN}
	M. d. M. González, Y.Y. Li and L. Nguyen,
	{\em Existence and uniqueness to a fully nonlinear version of the Loewner–Nirenberg problem}, 
	Commun. Math. Stat. 6 (2018), no. 3, 269–288.
	
	\bibitem{GW2}
	P. Guan and G. Wang,
	{\em A fully nonlinear conformal flow on locally conformally flat manifolds},
	J. Reine Angew. Math. 557 (2003), 219–238. 
	
	
	\bibitem{GS}
	M. J. Gursky and J. Streets,
	{\em A formal Riemannian structure on conformal classes and uniqueness for the $\sigma_2$–Yamabe problem}, 
	Geom. Topol. 22 (2018), no. 6, 3501–3573.
	
	\bibitem{GV}
	M. J. Gursky and J. A. Viaclovsky,
	{\em Prescribing symmetric functions of the eigenvalues of the Ricci tensor}, 
	Ann. of Math. (2) 166 (2007), no. 2, 475–531.
	
	\bibitem{HLL}
	Q. Han, X. Li and Y. Li,
	{\em Asymptotic expansions of solutions of the Yamabe equation and the $\sigma_k$-Yamabe equation near isolated singular points}, 
	Comm. Pure Appl. Math.. https://doi.org/10.1002/cpa.21943
	
	\bibitem{Han}
	Z.-C. Han,
	{\em Prescribing Gaussian curvature on $\mathbb{S}^2$}, 
	Duke Math. J. 61 (1990), no. 3, 679–703.
	
	
	\bibitem{He}
	W. He,
	{\em The Gursky–Streets equations}, 
	Math. Ann. (2020). https://doi.org/10.1007/s00208-020-02021-5
	
	\bibitem{JS1}
	F. Jiang and N. S. Trudinger,
	{\em Oblique boundary value problems for augmented Hessian equations II}, 
	Nonlinear Anal. 154 (2017), 148–173.
	
	\bibitem{JS2}
	F. Jiang and N. S. Trudinger,
	{\em Oblique boundary value problems for augmented Hessian equations I}, 
	Bull. Math. Sci. 8 (2018), no. 2, 353–411. 
	
	\bibitem{JS3}
	F. Jiang and N. S. Trudinger,
	{\em Oblique boundary value problems for augmented Hessian equations III}, 
	Comm. Partial Differential Equations 44 (2019), no. 8, 708–748.
	
	\bibitem{JLX}
	T. Jin, Y.Y. Li and J. Xiong,
	{\em The Nirenberg problem and its generalizations: A unified approach}, 
	Math. Ann. 369 (2017), no. 1-2, 109–151.
	
	\bibitem{LL1}
	A. Li and Y.Y. Li, 
	{\em On some conformally invariant fully nonlinear equations}, 
	Comm. Pure Appl. Math. 56 (2003), no. 10, 1416–1464.
	
	
	\bibitem{LL2}
	A. Li and Y.Y. Li, 
	{\em  On some conformally invariant fully nonlinear equations, II. Liouville, Harnack and Yamabe}, 
	Acta Math. 195 (2005), 117–154. 
	
	\bibitem{L89}
	Y.Y. Li, 
	{\em Degree theory for second order nonlinear elliptic operators and its applications}, 
	Comm. Partial Differential Equations 14 (1989), no. 11, 1541–1578.
	
	
	\bibitem{L09}
	Y.Y. Li, 
	{\em Local gradient estimates of solutions to some conformally invariant fully nonlinear equations}, 
	Comm. Pure Appl. Math. 62 (2009), no. 10, 1293–1326.
	
	
	\bibitem{LLL}
	Y.Y. Li, H. Lu and S. Lu,
	{\em On the $\sigma_2$-Nirenberg problem on $\mathbb{S}^2$ and related topics}, 
	in preparation.
	
	\bibitem{LN2}
	Y.Y. Li and L. Nguyen,	
	{\em A compactness theorem for a fully nonlinear Yamabe problem under a lower Ricci curvature bound}, 
	J. Funct. Anal. 266 (2014), no. 6, 3741–3771. 
	
	\bibitem{LN3}
	Y.Y. Li and L. Nguyen,	
	{\em Existence and uniqueness of Green's function to a nonlinear Yamabe problem}, 
	arXiv:2001.00993.
	
	\bibitem{LN4}
	Y.Y. Li and L. Nguyen,	
	{\em Solutions to the $\sigma_k $-Loewner-Nirenberg problem on annuli are locally Lipschitz and not differentiable}, 
	to appear in  J. Math. Study.
	
	\bibitem{LNW}
	Y.Y. Li, L. Nguyen and B. Wang, 
	{\em Comparison principles and Lipschitz regularity for some nonlinear degenerate elliptic equations}, 
	Calc. Var. Partial Differential Equations 57 (2018), no. 4, Paper No. 96, 29 pp.
	
	\bibitem{LNW1}
	Y.Y. Li, L. Nguyen and B. Wang, 
	{\em On the $\sigma_ {k}$-Nirenberg problem}, 
	arXiv:2008.08437.
	
	\bibitem{LW}
	Y.Y. Li and B. Wang,
	{\em Comparison principles for some fully nonlinear sub-elliptic equations on the Heisenberg group},
	 Anal. Theory Appl. 35 (2019), no. 3, 312–334. 
	
	\bibitem{LZ}
	Y.Y. Li and L. Zhang, 
	{\em  Liouville-type theorems and Harnack-type inequalities for semilinear elliptic equations}, 
	J. Anal. Math. 90 (2003), 27–87. 
	
	\bibitem{LZhu}
	Y.Y. Li and M. Zhu,
	{\em Uniqueness theorems through the method of moving spheres}, 
	Duke Math. J. 80 (1995), no. 2, 383–417.
	
	\bibitem{Liou1}
	J. Liouville,
	{\em Sur l'\'{e}quation aux diff\'{e}rences partielles $(\partial^2 \log\lambda/\partial u\partial v)\pm\lambda/2a^2=0$}, 
	J. de Math., 18 (1853), 71-72.
	
	\bibitem{Nitsche}
	J. C. C. Nitsche,
	{\em Elementary proof of Bernstein's theorem on minimal surfaces}, 
	Ann. of Math. (2) 66 (1957), 543–544.
	
	\bibitem{Obata}
	M. Obata,
	{\em The conjectures on conformal transformations of Riemannian manifolds}, 
	J. Differential Geometry 6 (1971/72), 247–258.
	
	\bibitem{Santos}
	A. S. Santos,
	{\em Solutions to the singular $\sigma_2$-Yamabe problem with isolated singularities}, 
	 Indiana Univ. Math. J. 66 (2017), no. 3, 741–790. 
	

	\bibitem{STW}
	W.-M. Sheng, N. S. Trudinger and X.-J. Wang,
	{\em The Yamabe problem for higher order curvatures},
	J. Differential Geom. 77 (2007), no. 3, 515–553. 
	
	\bibitem{Sui}	
	Z. Sui,
	{\em Complete conformal metrics of negative Ricci curvature on Euclidean spaces}, 
	J. Geom. Anal. 27 (2017), no. 1, 893–907.
	
	
	\bibitem{T2}
	N. S. Trudinger,
	{\em From optimal transportation to conformal geometry}, 
	in Geometric Analysis: In Honor of Gang Tian’s 60th Birthday, J. Chen, P. Lu, Z. Lu, and Z. Zhang, eds., Progress in Mathematics, Birkh\"{a}user, 2020.

\end{thebibliography}
\end{document}